\documentclass[12pt,a4paper]{amsart}

\usepackage{xspace}
\usepackage{enumerate}
\usepackage{latexsym,epsfig}
\usepackage{amssymb}
\usepackage{graphicx}

\newtheorem{theorem}{Theorem}[section]

\newtheorem{corollary}[theorem]{Corollary}
\newtheorem{remark}[theorem]{Remark}

\newtheorem{example}[theorem]{Example}

\newcommand{\N}{\mathbb{N}}
\newcommand{\C}{\mathbb{C}}
\newcommand{\B}{\mathcal{B}}

\newcommand{\op}{\operatorname{op}}

\newcommand{\w}{\operatorname{wk}}
\newcommand{\m}{\operatorname{m}}
\newcommand{\wlim}{\operatorname{wk*-lim}}

\newcommand{\si}{\sigma}

\def\S{{\mathcal S}}

\title[]{Schur null preserving maps}
\author{Ying-Fen Lin and Donal O'Cofaigh}

 \keywords{}

\address[Ying-Fen Lin, Donal O'Cofaigh]{Mathematical Sciences Research Centre, Queen's University Belfast, Belfast BT7 1NN, U.K.}
\email{y.lin@qub.ac.uk, docofaigh01@qub.ac.uk}

\begin{document}

\begin{abstract}
We provide a characterisation of Schur multiplicative maps on both finite and infinite dimensional matrix spaces, and show that every surjective Schur multiplicative contraction is automatically an isometry. We also generalise this result and provide a characterisation of Schur null preserving maps on $m\times n$ matrices and on Schur multipliers. 
\end{abstract}

\maketitle

\section{Introduction}

Let $m$ and $n$ be positive integers. The Schur (or the Hadamard) product of two $m$ by $n$ complex matrices $a, b \in M_{m,n}$ is defined as their entrywise product and is denoted by $a \ast b$. More precisely, for $a= [a_{ij}]$ and $b= [b_{ij}]$ in $M_{m,n}$, the Schur product of $a$ and $b$ is given by $a\ast b= [a_{ij} b_{ij}]$. Note that with the Schur product, the matrix space $M_{m,n}$ is a commutative (semi-simple) Banach algebra. The study of the Schur product is related to many pure and applied areas; see \cite{Horn}. Let $T: M_{m, n} \to M_{m, n}$ be a linear map. We say that $T$ is \emph{Schur null preserving} if $Ta \ast Tb = 0$ whenever $a \ast b= 0$. Similarly, we say that $T$ is \emph{Schur multiplicative} or a \emph{Schur homomorphism} if $Ta \ast Tb= T(a\ast b)$ for all $a, b \in M_{m, n}$. The same definitions can be given in $B(\ell^2)$, the space of all bounded linear operators on $\ell^2$, in which elements can be viewed as infinite matrices (indexed by $\N\times \N$). More precisely, the Hilbert space $\ell^2$ is separable and when equipped with its canonical orthonormal basis $\{e_i\}_{i= 1}^{\infty}$, an element $a$ in $B(\ell^2)$ can be viewed as a matrix $[a_{ij}]$, where $a_{ij}:= (a e_j, e_i) \in \C$ for all $i, j$. It is well-known that the Schur product is an internal operation in $B(\ell^2)$ which turns it into a semi-simple commutative Banach algebra with the usual operator norm and without unit (see e.g. \cite{Schur}). Some of its Banach algebra properties including the construction of its maximal ideal space were determined by Stout \cite{Stout} in a more general setting. 

Maps preserving certain algebraic, topological or geometric properties have been intensively studied in the literature. For instance, the classical Banach-Stone theorem characterises surjective isometries on continuous function spaces, it is well-known that such a map is a weighted homomorphism with modular one weight, moreover, it preserves the zero product of two continuous functions. Isometries and zero product (or disjointness) preserving operators have been studied in many settings, e.g. \cite{Kadison, RR, JP, FJ} for isometries, \cite{ABEV1, AEV, L} for zero product preservers and a survey \cite{Mol}, to name a few. Since the Schur product gives rise to a different algebraic structure, it is natural to seek for a characterisation of Schur multiplicative maps and Schur null preserving maps. While Schur multiplicative maps on matrices with some special properties such as $T(S) \subseteq S$ or $f(T(a))= f(a)$ for a given function $f$ on matrices were studied in \cite{LP, CLR, Poon}, we are interested in a general characterisation of Schur multiplicative maps and maps preserving Schur zero product. In particular, we obtain a complete characterisation of contractive and completely contractive Schur multiplicative maps. 

It is clear that every Schur multiplicative map preserves the Schur zero product; we will first give a general form of surjective Schur null preserving maps on matrix spaces and then obtain our characterisation of contractive Schur multiplicative maps both on $M_n$ and $B(\ell^2)$. As a corollary of our characterisation, we have that every surjective contractive Schur multiplicative map is an isometry, and every completely contractive one is a complete isometry.  We close the paper with a characterisation of bounded weak*-continuous Schur null preserving operator on the Schur multipliers of $B(\ell^2)$.

\section{Main results}

Let $m, n$ be positive integers, $J_{m,n}$ denotes the $m \times n$ matrix with all one entries and $a^t$ denotes the transpose of the matrix $a$ in $M_{m,n}$.

\begin{theorem}\label{NP on M_{m,n}}
Let $T: M_{m,n} \to M_{m,n}$ be a linear surjective map. Then $T$ is Schur null preserving if and only if $T(a)=T(J_{m,n})\ast a_{\rho}$ for all $a\in M_{m, n}$, where $\rho$ is a permutation on $\{1, 2, \ldots, m\}\times \{1, 2, \ldots, n\}$ such that $a_{\rho}:= [a_{\rho^{-1}(i, j)}]$ for $a= [a_{ij}]\in M_{m,n}$. 
\end{theorem}

\begin{proof}
For $(i, j)\in \{1, 2, \ldots, m\}\times \{1, 2, \ldots, n\}$, let $e_{i, j}$ be the matrix unit whose $ij$-entry is one and zero elsewhere. Note that if $T$ is Schur null preserving then for any two matrix units $e_{i,j}, \, e_{k,l}$ of $M_{m,n}$, we have that
\begin{eqnarray} \label{null product}
T(e_{i,j}) \ast T(e_{k,l})=0 \quad \text{when} \quad (i,j) \neq (k,l).
\end{eqnarray}

Given that $m,n$ are finite, for any pair $(i,j) \in \{1, \ldots, m\} \times \{1, \ldots, n\}$  the image of $e_{i,j}$ under $T$ can be expressed as a finite linear sum of basis elements in the image space. The property of $T$ being Schur null preserving is precisely the requirement that these images are disjoint - considered as the span of basis elements. 

As $T$ is surjective over a finite dimensional space it is necessarily bijective and this, coupled with equation \eqref{null product}, necessarily restricts $T$ to map each basis element to a scalar multiple of another basis element. The relationship so-defined between basis matrices in domain and image spaces yields a natural permutation $\rho:\{1, \ldots, m\} \times \{1, \ldots, n\} \to \{1, \ldots, m\} \times \{1, \ldots, n\}$ so that
\begin{align*}
T(e_{i,j}) = w_{i,j} \, e_{\rho(i,j)},
\end{align*}
where $w_{i, j}\in \C$. Let function $f$ be defined by $f(i, j):= w_{\rho^{-1}(i, j)}$ for all $(i, j) \in \{1, 2, \ldots, m\}\times \{1, 2, \ldots, n\}$.
We have
\begin{eqnarray*}
T\big(\sum_{i,j}a_{i,j}e_{i,j}\big)&=&\sum_{i,j} \, a_{i,j}T(e_{i,j}) \\
&=&\sum_{i,j} \, a_{i,j}w_{i,j}e_{\rho(i,j)} \\ 
&=&\sum_{i,j} \,w_{\rho^{-1}(i,j)}e_{i,j} \ast \sum_{i,j} \, a_{\rho^{-1}(i,j)}e_{i,j} \\ 
&=&\sum_{i, j} \, T(e_{i, j})\ast \sum_{i, j} a_{\rho^{-1}(i,j)}e_{i,j}\\
&=&T(J_{m,n})\ast [a_{\rho(i,j)}]\\
&=& f\ast a_{\rho}.
\end{eqnarray*}

To prove the converse, it is clear that if $\rho$ is a permutation on $\{1, \ldots, m\} \times \{1, \ldots, n\}$ and $T(J_{m,n})$ is a matrix with non-zero entries, the map given by $a\mapsto f\ast a_{\rho}$ preserves the Schur zero product. 
\end{proof}

Note that in case $T$ is a non-surjective Schur null preserving map on $M_{m,n}$, if the additional assumption is made that $T$ maps every basis element $e_{i,j}$ to a scalar multiple of another basis element then it is possible to retain a similar form of the characterisation. In this case, there may be more than one permutation $\rho$ which may work: each defined by choices made in the assignment of indices corresponding to basis elements in the kernel of the map. Furthermore the matrix $T(J_{m,n})$ will have at least one entry that is zero.

A Schur multiplicative map is necessarily Schur null preserving; moreover, we have $T(e_{i,j})=T(e_{i,j} \ast e_{i,j})=T(e_{i,j}) \ast T(e_{i,j})$. This means that each entry of $T(e_{i,j})$ must equal to either one or zero. Consequently, we have that $T(e_{i, j})= e_{\rho(i, j)}$ for some permutation $\rho$ on $\{1, 2, \ldots, m\} \times \{1, 2, \ldots, n\}$. Clearly, the map $a\mapsto a_{\rho}$ preserves the Schur product. 


\begin{corollary}\label{SM on M_{m,n}}
A surjective linear map $T$ on $M_{m, n}$ is Schur multiplicative if and only if $Ta= a_{\rho}$ for all $a\in M_{m, n}$, where $\rho$ is a permutation on $\{1, 2, \ldots, m\}\times \{1, 2, \ldots, n\}$. 
\end{corollary}

Note that the same result has previously been obtained by Clark, Li and Rastogi in ~\cite{CLR} with a different approach. Before moving into the infinite-dimensional matrices, we are interested to see if we could specify the permutations on the entries of the matrices with respect to our preservers. We say that a map $T$ is contractive if its operator norm $\|T\|_{\op}$ does not exceed one. We will see that with this additional assumption, we can achieve a more precise characterisation of Schur multiplicative maps. We first look at square matrices.


\begin{theorem}\label{thm: contractive SM on Mn}
Let $T$ be a linear surjective map on $M_n$. Then the following are equivalent:
\begin{enumerate}
\item \label{item:c1} the map $T$ is a Schur multiplicative contraction;
\item \label{item:c2} for all $a \in M_n$, either $T(a)= uav$ or $T(a)= ua^t v$ for some permutation unitaries $u, v$ in $M_n$.
\end{enumerate}
\end{theorem}

\begin{proof}
From (\ref{item:c2}) to (\ref{item:c1}), it is clear that $a\mapsto a^t$ preserves the Schur product and it is contractive. If $u$ and $v$ are  permutation unitaries, then $a\mapsto av$ ($a\mapsto ua$, respectively) is a map exchanging the rows (columns, respectively) of $a$ which is contractive and Schur multiplicative. 

From (\ref{item:c1}) to (\ref{item:c2}), let $\{e_{i,j}\}_{i, j= 1}^{n}$ be the canonical basis of $M_n$. Since $T$ is surjective and Schur multiplicative, from Corollary~\ref{SM on M_{m,n}} we have that $T(a)= a_{\rho}$ for some permutation $\rho$ on $\{1, 2, \ldots, n\}\times \{1, 2, \ldots, n\}$. We first claim that either $\rho(i, j)= (\pi(i), \sigma(j))$ or $\rho(i, j)= (\pi(j), \si(i))$ for some permutations $\pi, \si$ on $\{1, 2, \ldots, n\}$. Fix $i \in \{1, 2, \ldots, n\}$ and let $j\in \{1, 2, \ldots, n\}$ be arbitrary with $j\neq i$. Let $T(e_{i,i})= e_{s,t}$ and $T(e_{j,j})= e_{p,q}$ for some $s,t, p, q\in \{1, 2, \ldots, n\}$. Since $e_{i,i} \ast e_{j,j}= 0$ and $T$ is Schur multiplicative, we have that $e_{s,t} \ast e_{p,q}= 0$. Moreover, since $T$ is contractive, we have that $s \neq p$ and $t \neq q$. We claim that $T(e_{i,j})= e_{s,q}$ or $T(e_{i,j})= e_{p,t}$. Suppose otherwise, if $T(e_{i,j})= e_{k,l}$ for some $k, l \in \{1, \ldots, n\}$ such that $(k, l)\neq (s,q)$ and $(k, l)\neq (p, t)$, then it follows from the surjectivity of $T$ and the pigeonhole principle that there are $r_1, r_2$ with $r_1 \neq i, r_2\neq j$ such that  $T(e_{r_1,r_2})$ is on the $k$th row or $l$th column, which contradicts to the operator $T$ being contractive. Hence, we have $T(e_{i,j})= e_{s,q}$ or $T(e_{i,j})= e_{p,t}$. 
Let $\rho(i, i)=: (\pi(i), \si(i))$ for some functions $\pi$ and $\si$. Since $\rho$ is a permutation on $\{1, 2, \ldots, n\}\times \{1, 2, \ldots, n\}$, we have that $\pi, \si$ are permutations on $\{1, 2, \ldots, n\}$. We see that $\rho(i, j)= (\pi(i), \si(j))$ or $\rho(i, j)= (\pi(j), \si(i))$ for all $j\in \{1, 2, \ldots, n\}$.  
To complete the proof, we need to show that for any element $a= [a_{ij}]_{i,j}\in M_n$, it is either 
\begin{align*}
T([a_{ij}])= [a_{\pi(i)\si(j)}] \quad \text{or} \quad T([a_{ij}])= [a_{\si(i)\pi(j)}]^t.
\end{align*}
Take $k\in \{1, 2, \ldots, n\}$, assume that $T(e_{i,j})= e_{\pi(i)\si(j)}$, $T(e_{j,i})= e_{\pi(j)\si(i)}$ but $T(e_{k,j})= e_{\si(k)\pi(j)}^t= e_{\pi(j)\si(k)}$. We see that $\|e_{j,i}+ e_{k,j}\|= 1$ but 
\begin{align*}
\|T(e_{j,i}+ e_{k,j})\|= \|T(e_{j,i})+ T(e_{k,j})\|= \|e_{\pi(j)\si(i)}+ e_{\pi(j)\si(k)}\|> 1,
\end{align*}
a contradiction. Hence, we obtain $T(a)= uav$ in the first case, and in the second case we have $T(a)= u a^t v$, where $u$ and $v$ are the unitaries coming from the permutations $\pi$ and $\si$, respectively. 
\end{proof}

It is easy to see that the same arguments hold for rectangular matrices; therefore, a surjective linear map on $M_{m,n}$ is a Schur multiplicative  contraction if and only if $T(a)= uav$ for some permutation unitaries $u$ in $M_m$ and $v$ in $M_n$. As an immediate corollary we have the following.

\begin{corollary}\label{cor:auto isometry}
Every surjective contractive Schur multiplicative map on $M_{n}$ is an isometry.
\end{corollary}



In general we are interested in the characterisations of Schur null preserving maps defined on an infinite dimensional space. As we see, most of the arguments above hold for contractive Schur multiplicative maps on $B(\ell^2)$, however, in order to make sure that the maps $\pi$ and $\si$ defined as above are permutations on $\N$, we need a stronger assumption that $T$ being bijective. It is straightforward to have the following result. 

\begin{theorem}\label{thm: contractive bij SM on infinite-dim}
Let $T$ be a linear bijection on  $B(\ell^2)$. Then $T$ is contractive and Schur multiplicative if and only if there are permutation unitaries $u, v$ in $B(\ell^2)$ such that either $T(a)= uav$ or $T(a)= ua^t v$ for all $a\in B(\ell^2)$.
\end{theorem}

It is well-know that the transpose is not a completely bounded map (\cite{Paulsen}), we have the following characterisation of completely contractive, Schur multiplicative bijections on $B(\ell^2)$. 

\begin{corollary}\label{thm: c-contractive bij SM on infinite-dim}
Let $T$ be a linear bijection on $B(\ell^2)$. Then $T$ is completely contractive and Schur multiplicative if and only if $T(a)= uav$ for some permutation unitaries $u, v$ in $B(\ell^2)$. In particular, every completely contractive Schur multiplicative bijection on $B(\ell^2)$ is a complete isometry.
\end{corollary}



The following example indicates the need of the assumptions in Theorem~\ref{thm: contractive bij SM on infinite-dim}. 

\begin{example}
Let $T$ be an operator on $B(\ell^2)$ defined by
\[T([a_{ij}])=\begin{bmatrix}
a_{1 1} &a_{1 3} & a_{1 4} &a_{1 5}&\dots\\
a_{3 1}&a_{3 3}&a_{3 4}&a_{3 5}&\dots\\
a_{4 1}&a_{4 3}&a_{4 4}&a_{4 5}&\dots\\
a_{5 1}&a_{5 3}&a_{5 4}&a_{5 5}&\dots\\
\vdots& \vdots & \vdots & \vdots &  \ddots
\end{bmatrix}\] 

It is easy to check that $T$ is a contractive, surjective Schur multiplicative map which is not injective and clearly the operator $T$ does not carry the standard form. 
\end{example}





\section{On Schur multipliers}

It is natural to ask for an infinite dimensional version of Theorem~\ref{NP on M_{m,n}}. In this context, the matrix space $M_{m,n}$ is not replaced by $B(\ell^2)$ but by a space on which the Schur product behaves better in the infinite dimensional setting, namely, by the space $\S(\B)$ of all Schur multipliers of $B(\ell^2)$.


Let $\B= B(\ell^2)$, note that the Schur product is well-defined on $B(\ell^2)$ \cite{Paulsen}. By a \emph{Schur multiplier} of $B(\ell^2)$, we mean a map $m_{\psi}: \B \to \B$ of the form $a\mapsto \psi\ast a$, where $\psi= (\psi(i, j))_{i, j\in \N}$ is a fixed element. It follows that $\S(\B)$ is contained in $\ell^{\infty} (\N\times \N)$. Indeed, if $\{e_{i, j}\}_{i, j\in \N}$ is the canonical matrix units in $B(\ell^2)$ then we have that $|\psi(i, j)|= \|m_{\psi}(e_{i, j})\|_{\op}< \infty$ for all $i, j\in \N$. It is trivial to verify that $\S(\B)$ is a subalgebra of $\ell^{\infty} (\N\times \N)$ when the latter is equipped with the usual pointwise operations, and $\S(\B)$ is a dual Banach space. Moreover, by a result of R. R. Smith \cite{Smith}, each of these Schur multipliers has a completely bounded norm equal to its norm as linear map on $B(\ell^2)$, and it can be shown that $\S(\B)$ is a semi-simple commutative Banach algebra when equipped with the norm $\|\psi\|_{\m}:= \|m_{\psi}\|_{\op}$. On the other hand, it can be shown that $B(\ell^2)$ equipped with the multiplier norm $\|\cdot\|_{\m}$ is a (commutative) semi-simple Banach subalgebra of $\S(\B)$. We hence have the chain:
\begin{align*}
B(\ell^2) \subseteq \S(\B) \subseteq \ell^{\infty}(\N\times \N),
\end{align*}
where both inclusions are strict. For the first inclusion: Take the constant function $1$ having the value one on all $\N\times \N$. Clearly, $1$ is a Schur multiplier (in fact, $m_1$ is the identity transformation) but does not belong to $B(\ell^2)$. An example of a function which belongs to $\ell^{\infty}(\N \times \N)$ but not in $\S(\B)$ is the characteristic function $\chi_{\Delta}$ of the set $\Delta= \{(i, j): j\leq i\}$, see for example \cite{Davidson}.

 
  
\begin{theorem}\label{thm:SchurnullS}
Let $T: \S(\B) \to \S(\B)$ be a surjective, bounded, weak*-continuous linear operator. Then $T$ is a Schur null preserving map if and only if 
$T(a)= f \ast (a \circ {\rho})$, where $\rho: \N\times \N \to \N\times \N$ is a permutation and $f \in \S(\B)$. 
\end{theorem}

\begin{proof}
Take $n\in \N$, let $J_n:= \sum_{i, j= 1}^{n} e_{i, j}\in \S(\B)$; this is a projection onto the $n$ by $n$ matrices. Let $J$ be the identity with respect to the Schur product. By Theorem~\ref{NP on M_{m,n}}, we have that
\begin{align*}
T(J_n)= \sum_{i, j= 1}^{n} w_{i, j} e_{\rho(i, j)},
\end{align*}
where $w_{i, j}\in \C$ and $\rho: \{1, 2, \ldots, n\}\times \{1, 2, \ldots, n\} \to R_n\subseteq \N\times \N$ is a bijection. Let the function $f_n:= T(J_n)$ be given by
\begin{equation}
  f_n(i, j)= \begin{cases}
            w_{\rho^{-1}(i, j)} &\mbox{if } (i, j)\in R_n, \\
            0 &\text{otherwise}.
        \end{cases}
\end{equation}
Let $f: \N\times \N \to \C$ be defined by $f(i, j):= w_{\rho^{-1}(i, j)}$ for all $(i, j)\in \N\times \N$. We have that $f_n \to f$ pointwise as $n \to \infty$. On the other hand,  $J_n \xrightarrow[]{\w^*} J$ as $n\to \infty$, i.e. $J_n$ weak*-converges to $J$ and the operator $T$ is weak*-continuous, it follows that $f_n= T(J_n) \xrightarrow[]{\w^*} T(J)$. Moreover, from the assumption $T(J)$ is in $\S(\B)$, we have that $f= T(J)\in \S(\B)$. 
Now we want to claim that $T(a)= f\ast a_{\rho}$ for all $a\in \S(\B)$. To show this, let $a \in \S(\B)$, we have that $J_n \ast a \xrightarrow[]{\w^*} a$ as $n\to \infty$. From the weak*-continuity of $T$, we have that
\begin{eqnarray*}
T(a) &=& \wlim T(J_n\ast a)= \wlim \sum_{i, j= 1}^{n} a_{i, j} w_{i, j} e_{\rho(i, j)} \\
 &=& \wlim \sum_{(k,l)\in R_n} w_{\rho^{-1}(k, l)} a_{\rho^{-1}(k, l)} e_{k, l} \\
 &=:& \wlim~ (f_n \ast a^{(n)}_{\rho}),
\end{eqnarray*}
where $a^{(n)}_{\rho}:= J_n\ast a_{\rho}$, i.e., the projection of $a_{\rho}$ onto the $n\times n$ matrices.
On the other hand, we have that $(f_n\ast a^{(n)}_{\rho}) (i, j) \to (f\ast a_{\rho})(i, j)$ as $n\to \infty$ for all $(i, j)\in \N\times \N$. Hence, we have $f\ast a_{\rho}= T(a)\in \S(\B)$.

It is straightforward to show the operator $T$ with the form $T(a)= f\ast (a \circ {\rho})$ is Schur null preserving.
\end{proof}

Note that the function $f$ in Theorem~\ref{thm:SchurnullS} is bounded away from zero. 

\begin{remark}
It is worth noting that in general the operator $a \mapsto a_{\rho}$ defined on an infinite dimensional space may not be bounded, however, if we assume furthermore that $T(J)$ in Theorem~\ref{thm:SchurnullS} is invertible, then $a\mapsto a_{\rho}$ is bounded on $\S(\B)$.
\end{remark}

\end{document}